\documentclass[reqno,11pt]{amsart}

\pdfoutput=1

\usepackage{amsmath}
\usepackage{amsfonts}
\usepackage{amssymb}
\usepackage{enumerate}
\usepackage{amstext}
\usepackage{amsbsy}
\usepackage{amsopn}
\usepackage{bbm,amsthm}
\usepackage{amscd}
\usepackage{dsfont}
\usepackage{mathrsfs}
\usepackage[pdftex]{color}
\usepackage{amsxtra}
\usepackage{upref}
\usepackage{epstopdf}
\usepackage{graphicx,color}
\usepackage{hyperref}
\usepackage{longtable}
\usepackage{graphicx}
\usepackage[english]{babel}
\usepackage{todonotes}
\usepackage{cancel}

\DeclareFontFamily{OML}{rsfs}{\skewchar\font'177}
\DeclareFontShape{OML}{rsfs}{m}{n}{ <5> <6> rsfs5 <7> <8> <9> rsfs7
  <10> <10.95> <12> <14.4> <17.28> <20.74> <24.88> rsfs10 }{}
\DeclareMathAlphabet{\mathfs}{OML}{rsfs}{m}{n}

\newtheorem{theorem}{Theorem}
\newtheorem{lemma}[theorem]{Lemma}
\newtheorem{proposition}[theorem]{Proposition}

\theoremstyle{definition}

\newtheorem{conjecture}[theorem]{Conjecture}

\theoremstyle{remark}
\newtheorem{remark}[theorem]{\bf Remark}

\numberwithin{equation}{section}
\numberwithin{theorem}{section}

\newcommand{\intav}[1]{\mathchoice {\mathop{\vrule width 6pt height 3 pt depth  -2.5pt
\kern -8pt \intop}\nolimits_{\kern -6pt#1}} {\mathop{\vrule width
5pt height 3  pt depth -2.6pt \kern -6pt \intop}\nolimits_{#1}}
{\mathop{\vrule width 5pt height 3 pt depth -2.6pt \kern -6pt
\intop}\nolimits_{#1}} {\mathop{\vrule width 5pt height 3 pt depth
-2.6pt \kern -6pt \intop}\nolimits_{#1}}}

\newcommand{\intavl}[1]{\mathchoice {\mathop{\vrule width 6pt height 3 pt depth  -2.5pt
\kern -8pt \intop}\limits_{\kern -6pt#1}} {\mathop{\vrule width 5pt
height 3  pt depth -2.6pt \kern -6pt \intop}\nolimits_{#1}}
{\mathop{\vrule width 5pt height 3 pt depth -2.6pt \kern -6pt
\intop}\nolimits_{#1}} {\mathop{\vrule width 5pt height 3 pt depth
-2.6pt \kern -6pt \intop}\nolimits_{#1}}}

\newcommand{\ve}{\varepsilon}

\newcommand{\R}{\mathbb{R}}

\renewcommand{\P}[1]{{\mathbb{P}}\left[{#1}\right]}

\newcommand{\EE}[2]{{\mathbb{E}}\left[{#1}|{#2}\right]}

\DeclareMathOperator{\diam}{diam}

\begin{document}

\title[Pólya urns on hypergraphs]{Pólya urns on hypergraphs}

\author{Pedro Alves, Matheus Barros, Yuri Lima}
\date{\today}
\keywords{}
\thanks{Corresponding author: Yuri Lima}

\address{Pedro Alves, Instituto de Matemática e Estatística, Universidade de São Paulo (USP), Rua do Matão, 1010, Cidade Universitária, 05508-090. São Paulo -- SP, Brazil}
\email{pedroalvesqueiroz.qa@gmail.com}
\address{Matheus Barros, IMPA, Estrada Dona Castorina 110, CEP 22460-320, Rio de Janeiro, Brazil}
\email{matheusbarros201213@gmail.com}
\address{Yuri Lima, Instituto de Matemática e Estatística, Universidade de São Paulo (USP), Rua do Matão, 1010, Cidade Universitária, 05508-090. São Paulo -- SP, Brazil}
\email{yurilima@gmail.com}

\begin{abstract}
We study Pólya urns on hypergraphs and prove that, when the incidence
matrix of the hypergraph is injective, there exists
a point $v=v(H)$ such that the random process
converges to $v$ almost surely. We also provide a partial result when
the incidence matrix is not injective.
\end{abstract}

\maketitle

\section{Introduction}\label{sec:introduction}

In 1923, George Pólya introduced a simple urn model that has attracted attention
to probabilists. Nowadays called {\em classical Pólya urn} or simply {\em Pólya urn},
the simplicity of this random process with reinforcement allows for many variations
and adaptations to more complicated settings, some of which model concrete situations
in areas such as neuroscience, population dynamics, and social networks.

We are interested in Pólya urns on hypergraphs, as introduced in \cite[Section 9.2]{BBCL-15}.
Consider a hypergraph $H=(V,E)$ with
$V=[m]=\{1,2,\ldots,m\}$ and $|E|=N$, where every vertex $v\in V$ belongs to at least
one hyperedge $I\in E$. Place a bin at each vertex, and assume that on vertex $i$ the
bin contains initially $B_i(0)\geq 1$ balls. We consider the following random process of adding
$N$ balls to the bins at each step: if the numbers of balls after step $n-1$ are $B_1(n-1),\ldots,B_m(n-1)$,
step $n$ consists of adding, for each hyperedge $I\in E$, one ball to one of its vertices
according to the following probability:
$$
\P{i\text{ is chosen among }I\text{ at step }n}=\dfrac{B_i(n-1)}{\sum\limits_{j\in I}B_j(n-1)}\,\cdot
$$
We study the asymptotic behavior of the proportion of balls in the
bins of $H$, as the number of steps grows. More specifically, let $N_0=\sum_{i=1}^m B_i(0)$
denote the initial total number of balls, let $x_i(n)=\tfrac{B_i(n)}{N_0+nN}$, $i\in[m]$,
be the proportion of balls at vertex $i$ after step $n$, and let $x(n)=(x_1(n),\ldots,x_m(n))$.
We characterize the limiting behavior of $x(n)$ using a combinatorial information of the hypergraph,
given by its {\em incidence matrix} $I(H)$, see the definition in Section \ref{ss.notation}.
Let $\Gamma=\{(x_1,\ldots,x_m)\in\R^m:x_1+\cdots+x_m=0\}$.

\begin{theorem}\label{Thm-injective}
Let $H$ be a finite hypergraph. If the restriction of $I(H)$ to $\Gamma$ is injective, 
then there exists a point $v=v(H)$ such that $x(n)$ converges to $v$ almost surely.
\end{theorem}

Observe that $v=v(H)$ is a deterministic point that only depends on $H$ and not on the 
random process. Our impression, based on many examples, is that the above injectivity
condition usually holds, unless $H$ exhibits some symmetry on its structure.

This result is more general than the one stated in the abstract, since it only requires the restriction 
$I(H)\restriction_\Gamma$ to be injective. 
We also provide a partial answer when the restriction
is not injective. Let $\Delta^{m-1}=\{(x_1,\ldots,x_m)\in\R^m:x_i\geq 0 \text{ and }x_1+\cdots+x_m=1\}$
be the $(m-1)$--th dimensional simplex. Let $K={\rm ker}(I(H)\restriction_{\Gamma})$.

\begin{theorem}\label{thm-not-injective}
Let $H$ be a finite hypergraph. There exists a closed connected subset
$\mathscr{J} = \mathscr{J}(H)$ of an affine subspace of $\R^m$ parallel to $K$ such that
the limit set of $x(n)$ is contained in $\mathscr J$ almost surely. Furthermore, conditioned on the
event that the limit set of $x(n)$ does not intersect 
$\partial\Delta^{m-1}$, $x(n)$ converges to a point of $\mathscr{J}$ almost surely.
\end{theorem}

In particular, if the limit set of $x(n)$ does not intersect $\partial\Delta^{m-1}$ almost surely,
then $x(n)$ converges to a point of $\mathscr{J}$ almost surely. We observe that it is possible
that $\lim x(n)$ exists and belongs to $\partial\Delta^{m-1}$ with positive probability.
For instance, if $H$ has a leaf and $I(H)$ is injective, then the point $v(H)$ given by Theorem 
\ref{Thm-injective} belongs to $\partial\Delta^{m-1}$, see Lemma \ref{lemma-leaf}. However, we could not find an example
where $\lim x(n)$ exists and belongs to $\partial\Delta^{m-1}$ with probability in $(0,1)$.

There are many motivations to consider Pólya urns on hypergraphs. On a practical level, while
Pólya urns on graphs can model networks where interactions occur in pairs,
Pólya urns on hypergraphs can model multivariate interaction.
The interest for models with multivariate interaction 
has been growing due to their ability to address situations with strong correlation 
in diferent levels. An example of this growth is the extensive study of hypergraph network 
in recent years, see e.g. the book \cite{DG23}. As a consequence, hypergraph networks have
found a wide variety of applications in different areas, such as medicine for survival prediction based on images \cite{DLZG20}
and drug discovery \cite{RJY+21}, as well as computer science
for visual classification \cite{YTW12, GWTJD12, ZLZJG18, DZFZJDG23}. 
On a theoretical level, as seen in this article, Pólya urns on hypergraphs provide new problems
and difficulties that are not present in its graph version. This allows for a unification of the theory,
thus stressing the main features of the topic.

Another motivation for studying Pólya urns on hypergraphs relates to a hypergraph
version of the example introduced in \cite{BBCL-15}, as we now explain.
Consider a market containing
$m$ companies and $N$ products. Each product is sold by a subset of the companies.
We can represent the market by a hypergraph $H=(V,E)$ where
$V$ is the set of companies and each $I\in E$ represents a product. 
The companies try to use their size and reputation to boost their sales.
Assuming that each client wants to buy one unit of each of the $N$ products at a moment,
his/her choice of a company that sells the product $I$ corresponds to adding a ball to the
respective company of the hyperedge $I$. Therefore, the Pólya urn on $H$ 
describes in broad strokes the long-term evolution of the market.

We can also enhance the model studied in \cite{SP-2000}, which considers
a network of agents that play repeated games in pairings and then improve their skill by
gaining experience. If the games are played by more than two agents according
to an underlying hypergraph, then the experience of the agents can be represented by 
a Pólya urn on a hypergraph, where the numbers of balls in the bins represent the skill levels
gained by the agents.

There are many recent variant and developments in the theory of Pólya urn schemes
and related topics,
see e.g. \cite{ACG-19,Aletti-20,AGS-22,Antunovic-16,BC-22,CDLM-19,CH-21,CJ-22,HHK-21,HHK-23,
Hofstad-16,KMS-22,RPP-22,Sahasrabudhe-16}.
Prior to its introduction in \cite{BBCL-15}, Pemantle
considered a Pólya urn model on $V=[m]$ with 
a single hyperedge $\{1,2,\ldots,m\}$, which is the same as a classical Pólya urn
with balls of $m$ colors \cite{Pemantle-92}. Up to this particular setting and to the authors' knowledge,
the present article gives the first general result for Pólya urns on hypergraphs.

The case of Pólya urns on graphs introduced in \cite{BBCL-15} actually considers,
for a fixed $\alpha>0$, the model of adding balls to the bins
with probability proportional to the $\alpha$--th power of its current number of balls.
Here, we only consider the case $\alpha=1$, which we call the {\em linear case}.
For graphs, the limiting behavior in the linear case has been completely
solved in the series of works \cite{BBCL-15,CL-14,Lim-16}. The final result is the following.

\begin{theorem}[\cite{BBCL-15,CL-14,Lim-16}]\label{thm-graphs}
Let $G$ be a finite connected graph.
\begin{enumerate}[{\rm (a)}]
\item If $G$ is not balanced bipartite, then there is a point $v=v(G)$
such that $x(n)$ converges to $v$ almost surely.
\item If $G$ is balanced bipartite, then there is a closed interval $\mathfs J=\mathfs J(G)$
such that $x(n)$ converges to a point of $\mathfs J$ almost surely.
\end{enumerate}
\end{theorem}

Above, the notion of a balanced bipartite graph means that the vertex set has a partition 
$V=A\cup B$ with $|A|=|B|$ such that every edge of $G$ has one endpoint in $A$ and one in $B$.

The proofs of Theorems \ref{Thm-injective} and \ref{thm-not-injective} follow very closely
the approach used to prove Theorem \ref{thm-graphs}, which consists of
writing $x(n)$ as a {\em stochastic approximation algorithm}, i.e.
as a small perturbations of a vector field $F$, see Section~\ref{s-saa}. By the work of 
Benaïm \cite{Benaim-96,Benaim-99}, it is possible to relate the limiting behavior of 
$x(n)$ with dynamical properties of $F$. 

The present article uses in great extent the tools developed in \cite{BBCL-15,CL-14,Lim-16}. 
One contribution of our work is to identify the combinatorial object that is related to the
equilibria set of $F$, which is the {\em incidence matrix} of the hypergraph, 
see Section \ref{ss.notation} for the definition. This notion was not present in the previous cited works,
and its introduction to the subject greatly clarifies the relationship between $x(n)$ and 
the vector field $F$. When restricted
to $\Gamma$ (which is the tangent space of the simplex $\Delta^{m-1}$), the incidence matrix provides 
the space of candidates for limits of $x(n)$, which is a subset $\mathfs J$ of an affine subspace
parallel to $K={\rm ker}(I(H)\restriction_{\Gamma})$.
We point out that our result applies, in particular,
to Pólya urns on graphs, in which the role of the incidence matrix was not evident, and rather
the adjacency matrix was used.

Another contribution of our work is to show that,
for points in the interior of $\mathfs J$, the behavior of $F$ in the transverse
direction to $\mathfs J$ is contractive, see Lemma \ref{lemma-neg-eig}.
The proof of this fact for graphs was wrongly obtained in \cite{Lim-16},
so we take the chance to correct this issue. 

We believe that Theorem \ref{thm-not-injective} can be improved to show that
$x(n)$ converges almost surely, even when its limit set is contained in $\partial\Delta^{m-1}$
with positive probability, hence we state the following conjecture.

\begin{conjecture}\label{conj:hypergraph-convergence}
For any finite hypergraph, the random process $x(n)$ converges almost surely.
\end{conjecture}

\subsection{Notation}\label{ss.notation}

We consider a hypergraph $H=(V,E)$ where $V=[m]=\{1,2,\ldots,m\}$ and $|E|=N$.
An element $I\in E$ is called an {\em hyperedge}. We assume that
every $i\in[m]$ belongs to at least one hyperedge.

\medskip
\noindent
{\sc Simplex $\Delta^{m-1}$:} We let 
$$
\Delta^{m-1}=\{(x_1,\ldots,x_m)\in\R^m:x_i\geq 0 \text{ and }x_1+\cdots+x_m=1\}$$
denote the $(m-1)$--dimensional simplex, which is a manifold with boundary.

\medskip
\noindent
{\sc Tangent space $\Gamma$:} We let
$$
\Gamma=\{(x_1,\ldots,x_m)\in\R^m:x_1+\cdots+x_m=0\},
$$
which is the tangent space of $\Delta^{m-1}$ at every $v\in {\rm int}(\Delta^{m-1})$
(we take the relative interior). 

\medskip
Given a hyperedge $I\in E$ and $v\in \R^m$, we write
$$
v_I:=\sum_{i\in I}v_i.
$$
Recall that we start with $B_1(0),\ldots,B_m(0)\geq 1$ balls
in the vertices $1,\ldots,m$ respectively, and that $N_0=\sum_{i=1}^m B_i(0)$ denotes
the initial total number of balls. After $n$ steps,
  $B_1(n),\ldots,B_m(n)$ denote the number of balls in the vertices $1,\ldots,m$ respectively. 
The total number of balls after step $n$ is $N_0+nN=\sum_{i=1}^m B_i(n)$.

\medskip
\noindent
{\sc Incidence matrix $I(H)$:} The {\em incidence matrix} of $H$ is the matrix
$I(H)$ with dimensions $N\times m$, indexed by $E\times V$,
whose entry $(I,i)$ is 1 if $i\in I$ and 0 otherwise.

\medskip 
\noindent
{\sc Subspace $K$:} We let $K={\rm ker}(I(H)\restriction_\Gamma)$.

\subsection{Some examples}\label{ss-examples}

Every solid $S$ can be viewed as an hypergraph $H$, whose hyperedges are the faces
of $S$. For example, a tetrahedron is a hypergraph with $V =\{1,2,3,4\}$ and
hyperedges
$E = \{\{1,2,3\},\{1,2,4\}, \{1,3,4\},\{2,3,4\}\}$. For the platonic solids, the kernel of $I(H)$ coincides
with $K$, and their dimensions are:
\begin{center}
\begin{tabular}{ |l||c| } 
\hline
Platonic solid & \hspace{.6cm} $\dim K$ \hspace{.6cm} \\
\hline \hline
Tetrahedron & $0$ \\ \hline
Cube &  4 \\ \hline
Octahedron & 2 \\ \hline
Icosahedron & $0$ \\ \hline
Dodecahedron & 8 \\ \hline
\end{tabular}
\end{center}
In all cases, the uniform measure $\left(\tfrac{1}{m},\ldots,\tfrac{1}{m}\right)$ 
is an equilibrium. For the tetrahedron and the icosahedron,
the process $x(n)$ converges to the uniform measure almost surely.
See a detailed discussion about the cube in Section \ref{ss-concluding}.

\section{Stochastic approximation algorithms}\label{s-saa}

P\'olya urns on graphs are examples of stochastic approximation algorithms \cite[\S 2]{BBCL-15}.
In this section, we show that the same occurs to Pólya urns on hypergraphs.
From now on, we fix a hypergraph $H$.

\medskip
\noindent
{\sc Stochastic approximation algorithm:} A {\em stochastic approximation algorithm} is a discrete time process 
$\{x(n)\}_{n\geq 0}\subset\mathbb R^m$
of the form
$$
x(n+1)-x(n)=\gamma_{n+1}G(x(n), \xi(n+1))
$$
where $\{\gamma_n\}_{n\ge 1}$ is a sequence of nonnegative scalars, $G:\mathbb R^m\times \R^m\to\mathbb R^m$ is a measurable function that characterizes the algorithm,
$\{\xi(n)\}_{n\geq 1}\subset\mathbb R^m$ is a sequence of random inputs 
and $G(x(n), \xi(n+1))$ is an
observable.

\medskip
Above, $\{\gamma_n\}_{n\ge 1}$ is a deterministic sequence.
For simplicity, we just write $x(n)$ for the sequence $\{x(n)\}_{n\geq 0}$.
Let $\mathfs F_n$ be the sigma-algebra generated by the process up to step $n$,
i.e. generated by $x(0),\xi(1),\ldots,\xi(n)$. 
Recall that we are fixing a hypergraph $H$, and letting $\{x(n)\}_{n\geq 0}$ denote the discrete process of 
Pólya urns on $H$.

\begin{lemma}\label{lemma.SAA}
The process $x(n)$ is a stochastic approximation algorithm of the form
$$
x(n+1) - x(n) = \gamma_{n+1}(F(x(n)) + u_{n+1})
$$
where $F: \R^m \to \R^m$ is a vector field, $\{\gamma_n\}_{n\geq 1}$ is the sequence of
nonnegative scalars, and $\{u_n\}_{n\geq 1}$ is a sequence of random inputs such that $\mathbb{E}[u_{n+1}|\mathscr{F}_n] = 0$.
\end{lemma}

\begin{proof}
Consider a family of $0,1$--valued random variables $\{\delta_{I \to i}(n+1)\}_{I\in E,i\in I}$ such that:
\begin{enumerate}[$\circ$]
\item If $I_1\neq I_2$, then $\delta_{I_1 \to i_1}(n+1)$ and $\delta_{I_2\to i_2}(n+1)$ are independent 
for all $i_1\in I_1,i_2\in I_2$;
\item $\sum_{i \in I}\delta_{I \to i}(n+1)=1$ for every $I\in E$;
\item $\EE{\delta_{I \to i}(n+1)}{\mathscr{F}_n}=\dfrac{x_i(n)}{x_I(n)}$ for every $i\in I$.
\end{enumerate}
The random variable $\delta_{I \to i}(n+1)$ represents the addition or not, at step $n+1$,
of the ball thrown at the hyperedge $I$ to the vertex $i$.
Then $C_i(n+1) = \sum_{I\ni i} \delta_{I \to i}(n+1)$
 is the number of balls added to the vertex $i$ at step $n+1$. Therefore
\begin{align*}
&\ x_i(n+1)-x_i(n)=\frac{B_i(n)+C_i(n+1)}{N_0 +(n+1)N}-\frac{B_i(n)}{N_0 + nN}\\
&=\frac{B_i(n)}{N_0 +(n+1)N}\left(1-\frac{N_0 +(n+1)N}{N_0 +nN}\right)+\frac{C_i(n+1)}{N_0 +(n+1)N}\\
&=\frac{N}{N_0 +(n+1)N}\left(-\frac{B_i(n)}{N_0 +nN}\right)+\frac{C_i(n+1)}{N_0 +(n+1)N}\\
&=\frac{1}{\frac{N_0}{N} +(n+1)}\left(-x_i(n)+\frac{1}{N}C_i(n+1)\right).
\end{align*}
Letting $\gamma_{n+1}=\frac{1}{\frac{N_0}{N} +(n+1)}$, $\xi_i(n+1)=\frac{1}{N}C_i(n+1)$ and
$\xi(n+1) = (\xi_1(n+1),\ldots,\xi_m(n+1))$, we get that
$$
x(n+1)-x(n) = \gamma_{n+1}(-x(n)+\xi(n+1)).
$$
Hence $x(n)$ is a stochastic approximation algorithm.
To finish the proof, we adjust $\xi(n+1)$ to have zero expectation. 
We have
$$
\EE{\xi_i(n+1)}{\mathscr{F}_n}=\mathbb{E}\left[\frac{1}{N}C_i(n+1)\Biggr |\mathscr{F}_n\right]
=\frac{1}{N}\sum_{I\ni i}\EE{\delta_{I\to i}(n+1)}{\mathscr{F}_n}
=\frac{1}{N}\sum_{I\ni i}\frac{x_i(n)}{x_I(n)}\,\cdot
$$        
Defining the vector field $F=(F_1,\ldots,F_m)$ by 
$F_i(v)=-v_i+\frac{1}{N}\sum\limits_{I\ni i}\frac{v_i}{v_I}$ and the vector
$u_{n+1}=\xi(n+1)-\EE{\xi(n+1)}{\mathscr{F}_n}$, we obtain that
$$
x(n+1)-x(n) = \gamma_{n+1}\left[F(x(n))+u_{n+1}\right],
$$
with $\EE{u_{n+1}}{\mathfs F_n}=0$. This concludes the proof.
\end{proof}

Since $F$ involves fractions which can have zero denominator, we now give its proper definition.
Fix $c<\tfrac{1}{N}$.

\medskip
\noindent
{\sc Set $\Delta$:} We let $\Delta=\{(x_1,\ldots,x_m)\in\Delta^{m-1}:x_I \geq c,\forall I \in E\}$.

\medskip
\noindent
{\sc Vector field $F$:} We define $F=(F_1,\ldots,F_m):\Delta\to\R^m$ where
$$
F_i(v)=-v_i+\frac{1}{N}\sum_{I\ni i}\frac{v_i}{v_I}\,\cdot
$$

\medskip
Note that $F(v)\in \Gamma$ for every $v\in\Delta$ and that 
$F$ is Lipschitz. Now we prove that $F$ defines a semiflow. For that, we consider the ODE
$$
\frac{dv(t)}{dt}=F(v(t))
$$
and prove that $\Delta$ is positively invariant under this ODE (we will also say that $\Delta$
is positively invariant under $F$). Given $I\in E$, we have
$$
\frac{dv_I}{dt}=\sum_{i\in I}F_i(v)= \sum_{i\in I}\left(-v_i+\frac{1}{N}\sum_{J\ni i}\frac{v_i}{v_J}\right)
\geq -v_I+\frac{1}{N}\sum_{i\in I}\frac{v_i}{v_I}
=-v_I+\frac{1}{N}\,\cdot
$$
If $v\in\partial\Delta$ with $v_I=c$, then
$$ 
\frac{dv_I}{dt}=-c+\frac{1}{N}>0
$$
and so $F$ points inwards $\Delta$ at its boundary.

\subsection{The vector field $F$ is gradient-like}
Now we prove that $F$ is gradient-like. The proof is similar
to the one for Pólya urns on graphs \cite[Lemma 4.1]{BBCL-15}.
Let us recall some definitions.

\medskip
\noindent
{\sc Equilibria set:} A point $v\in\Delta$ is called an {\em equilibrium} for $F$ if $F(v)=0$. 
The {\em equilibria set} of $F$ is denoted by $\Lambda=\{v\in\Delta:v\text{ is an equilibrium}\}$.

\medskip
\noindent
{\sc Strict Lyapunov function:} A continuous map $L:\Delta\to\mathbb R$ is called a
{\it strict Lyapunov function for $F$} if it is strictly monotone along any integral curve of $F$ outside of $\Lambda$.
In this case, we call $F$ {\em gradient-like}.

\medskip
Let $L:\Delta\rightarrow\mathbb R$ be the function
$$
L(v_1,\dots,v_m)=-\sum_{i=1}^{m}v_i+\frac{1}{N}\sum_{I\in E}\log v_I=-1+\frac{1}{N}\sum_{I\in E}\log v_I.
$$
This is the version for hypergraphs of the strict Lyapunov function used for Pólya urns on graphs
in \cite{BBCL-15}. The next lemma is \cite[Lemma 4.1]{BBCL-15} for hypergraphs.

\begin{lemma}
The function $L$ is a strict Lyapunov function for $F$. Therefore, $F$ is gradient-like.
\end{lemma}

\begin{proof}
By direct calculation, 
\begin{equation}\label{partial-derivative}
\frac{\partial L}{\partial v_i}=-1+\frac{1}{N}\sum_{I\ni i}\frac{1}{v_I}
\end{equation}
and so 
$$
\frac{dv_i}{dt}=F_i(v(t))=v_i\left(-1+\frac{1}{N}\sum_{I\ni i}\frac{1}{v_I}\right)=v_i\frac{\partial L}{\partial v_i}\,\cdot
$$ 
If $v=(v_1(t),\dots,v_m(t))$, $t\geq 0$, is an integral curve of $F$, then 
$$
\frac{d}{dt}(L\circ v)=\sum_{i=1}^{m}\frac{\partial L}{\partial v_i}\frac{dv_i}{dt}=\sum_{i=1}^{m}v_i\left(\frac{\partial L}{\partial v_i}\right)^2\geq 0.
$$ 
The equality occurs if and only if $v_i\left(\frac{\partial L}{\partial v_i}\right)^2=0$ for all $i \in [m]$.
This holds if and only if $F(v)=0$, hence $L$ is a strict Lyapunov function for $F$.
\end{proof}

\subsection{Relation between $x(n)$ and $F$}\label{ss-relation-x-F}

In the sequel, we state a result of Benaïm
that allows to relate asymptotic properties of $x(n)$ with
dynamical properties of $F$. 
The theorem is not stated in its whole generality, but instead specialized to our situation.
 
\begin{theorem}[\cite{Benaim-96,Benaim-99}]\label{thm:limit set theorem}
Let $F:\R^m \to \R^m$ be a gradient-like continuous vector field with unique integral curves, 
let $\Lambda$ be its equilibria set, let $L$ be a strict Lyapunov function,
and let $x(n)$ be a solution to the recursion
$$
x(n+1)-x(n)=\gamma_{n+1}[F(x(n))+u_{n+1}],
$$
where $\gamma_n$ is a decreasing sequence satisfying
$\lim\limits_{n\to\infty}\gamma_n=0$ and $\sum\limits_{n\geq 1}\gamma_n = \infty$,
and $u_n\in\R^m$. Assume that:
\begin{enumerate}[{\rm (1)}]
        \item the sequence $x(n)$ is bounded,\\
        \item for each $T>0$,
        $$
        \lim_{n \to \infty}\left(\sup_{\{k: 0\leq \tau_k-\tau_n\leq T\}}\left\lVert \sum_{i=n}^{k-1}\gamma_i u_i \right\rVert \right)=0,
        $$
        where $\tau_n = \sum\limits_{i=0}^{n-1}\gamma_i$, and
        \item $L(\Lambda)\subset \R$ has empty interior.
\end{enumerate}
Then the limit set of $x(n)$ is a compact connected subset of $\Lambda$.
\end{theorem}

Above, the $u_n$'s are fixed. Not every realization of the
Pólya urn on hypergraphs satisfies the above conditions, but it does almost surely,
as we will show in the next proposition.
First, we need some notation.

\medskip
\noindent
{\sc Faces of $\Delta$:} For each $S \subset [m]$, we let
$$
\Delta_S = \{v \in \Delta : v_i = 0 \text{ iff } i\notin S\}
$$
denote the {\em face} of $\Delta$ determined by $S$.
Each $\Delta_S$ is a manifold with boundary, positively invariant under $F$.

\medskip
\noindent
{\sc $S$--singularity for $L$:} We call $v \in \Delta_S$ an {\em $S$--singularity for $L$}
or simply {\em $S$--singularity} if
$$
\dfrac{\partial L}{\partial v_i}(v) = 0 \text{ for all } i \in S.
$$
Let $\Lambda_S$ denote the set of $S-$singularities. Since $F_i(v)=v_i\tfrac{\partial L}{\partial v_i}(v)$,
we have $\Lambda=\bigcup\limits_{S\subset [m]}\Lambda_S$.

\begin{proposition}\label{prop: limit set}
Let $x(n)$ be the random process of Pólya urn on $H$. 
Then the limit set of $x(n)$ is a compact connected subset of $\Lambda$
almost surely.
\end{proposition}

\begin{proof}
The sequence $\gamma_{n+1}=\frac{1}{\frac{N_0}{N} +(n+1)}$
is decreasing with $\lim\limits_{n\to\infty}\gamma_n=0$ and 
$\sum\limits_{n\geq 0}\gamma_n = \infty$.
Note that $x(n)$ is bounded and so (1) always holds.
We show that (2) holds almost surely. For each $n\geq 1$, let
$$
M_n = \sum_{i=1}^{n}\gamma_i u_i.
$$
The sequence of random variables $\{M_n\}_{n\geq 1}$ is a martingale adapted 
to the filtration $\{\mathscr{F}_n\}_{n\geq 1}$:
$$
\EE{M_{n+1}}{\mathscr{F}_n} = \sum_{i=1}^{n}\gamma_i u_i + \EE{\gamma_{n+1}u_{n+1}}{\mathscr{F}_n}
=\sum_{i=1}^{n}\gamma_i u_i = M_n.
$$
Furthermore, since $|u_n|\leq 1$, we have 
$$
\sum_{i=1}^n\EE{\|M_{i+1}-M_i \|^2}{\mathfs F_i} \leq \sum_{i=1}^n \gamma_{i+1} ^2 \leq \sum_{i\geq 1}\gamma_i ^2
< \infty.
$$
This latter estimate implies that $\{M_n\}_{n\geq 1}$ converges almost surely to
a finite random vector, see e.g. \cite[Theorem 5.4.9]{Durrett-book}. In particular, $\{M_n\}_{n\geq 1}$
is a Cauchy sequence almost surely, and so condition (2) holds almost surely.

It remains to check condition (3). Let $S\subset[m]$. The restriction $L\restriction_{\Delta_S}$ 
is a $C^{\infty}$ function,
thus, by the Sard theorem, $L(\Lambda_S)$ has zero Lebesgue measure. This implies
that $L(\Lambda)=\bigcup\limits_{S\subset[m]}L(\Lambda_S)$ has zero Lebesgue measure as well.
In particular, it has empty interior.
\end{proof}

We note that $L$ is a (not necessarily strictly) concave function, hence so is
its restriction $L\restriction_{\Delta_S}$. This implies that $\Lambda_S$ is equal to the set
of global maxima of $L\restriction_{\Delta_S}$.

\section{Unstable and non-unstable equilibria}

In this section, we restrict the possible limits of $x(n)$.
The idea, following \cite{Pemantle-92}, is very similar to the one developed
in \cite{BBCL-15}, and consists of showing that there is zero probability of converging
to an unstable equilibrium (we will define this notion shortly). The sole difference is that,
contrary to the referred works, here we can have an uncountable number of unstable
equilibria. 
Recall that $K={\rm ker}(I(H)\restriction_{\Gamma})$. The next lemma characterizes $\Lambda_S$
when it is non-empty. This result will not be used here, but since it is short and we believe it
will be useful for later work, we have decided to include it.

\begin{lemma}\label{lemma.equil.face}
Let $S\subset[m]$ and assume that $\Lambda_S\neq\emptyset$. For
every $v\in \Lambda_S$, it holds that
$$
\Lambda_S=\Delta_S\cap(v+K)=\{v+w\in \Delta_S:w\in K\}.
$$
\end{lemma}

In other words, $\Lambda_S$ is the part of the affine subspace $v+K$ that intersects $\Delta_S$.
Therefore, $\Lambda$ is contained in a finite union of translates of $K$.

\begin{proof}

Fix $v \in \Lambda_S$ and let $w \in \Delta_S\cap(v+K)$. Since $v-w \in K$,
we have $v_I=w_I$ for all $I\in E$, and so
$\frac{\partial L}{\partial v_i}(w) = \frac{\partial L}{\partial v_i}(v) = 0$ for every $i \in S$.
Therefore, $w \in \Lambda_S$. 

Conversely, let $w \in \Lambda_S$. As observed at the end of Section \ref{ss-relation-x-F},
$L\restriction_{\Delta_S}$ attains its global maxima in every point of $\Lambda_S$. 
Hence, we get $L(w)=L(v)$ and so
\begin{equation}\label{eq:log-equality}
\sum_{I \in E} \log v_{I} = \sum_{I \in E} \log w_{I}
    .
\end{equation}
Since the map $V\in \R^N_{+} \mapsto \sum_{I \in E} \log(V_{I}) $ is strictly concave,
equality~\eqref{eq:log-equality} implies that $v_{I} = w_{I} $ for all $I \in E$, and
so $v - w \in K$, which concludes the proof. 
\end{proof}

Now we analyze the dynamical type of equilibria.

\subsection{Unstable equilibria}

Let $S\subset[m]$, and fix $w\in\Lambda_S$. Consider the derivative $DF(w):T_w\Delta\to T_w\Delta$.
In coordinates $v_1,\ldots,v_m$, this linear transformation is represented by
$JF(w)=\left(\frac{\partial F_i}{\partial v_j}\right)_{i,j}$ with:
\begin{equation}\label{jacobian}
\frac{\partial F_i}{\partial v_j}=\left\{\begin{array}{ll}
w_i\dfrac{\partial ^2L}{\partial v_i\partial v_j}&\text{if }i\not=j,\\
&\\
\dfrac{\partial L}{\partial v_i}+w_i\dfrac{\partial ^2L}{\partial v_i^2} &\text{if }i= j.\\
\end{array}\right.
\end{equation}
Without loss of generality, assuming that $S=\{k+1,\ldots, m\}$, we have that 
\begin{align}\label{definition jacobian}
JF(w)=\left[
\begin{array}{cc}
A & 0\\
C & B \\
\end{array}
\right]
\end{align}
where $A$ is a $k\times k$ diagonal matrix with diagonal entries $a_{ii}=\frac{\partial L}{\partial v_i}(w)$, $i\in[k]$.
The spectrum of $JF(w)$ is equal to the union of the spectra of $A$ and $B$. Introducing the inner
product $(x,y)=\sum_{i=k+1}^m x_iy_i/w_i$, we have $(Bx,y) = \langle Dx, y\rangle$,
where $D$ is the Hessian matrix of $L$ restricted to the coordinates $v_{k+1},\ldots,v_m$ and 
$\langle \cdot,\cdot\rangle$ is the canonical inner product. Since $D$ is symmetric,
$B$ is self-adjoint. Since $L$ is concave, $B$ is negative semidefinite. 
Therefore, the eigenvalues of $B$ are real and nonpositive, and so
$JF(w)$ has a real positive eigenvalue if and only if $a_{ii}>0$ for some $i\in[k]$,
which justifies the following definition.

\medskip
\noindent
{\sc Unstable equilibrium:} We call $w\in \Lambda$ an {\em unstable equilibrium}
if there exists $i\in [m]$ such that $w_i=0$ and $\frac{\partial L}{\partial v_i}(w) > 0$.
If $w$ is not unstable, we call $w$ a {\em non-unstable equilibrium}.

\medskip
By definition, every equilibrium in the interior of $\Delta$ is non-unstable.
The next lemma, proved in Appendix \ref{sec: non-convergence}, is the version for hypergraphs of \cite[Lemma 5.2]{BBCL-15}.
We let ${\rm Lim}\{x(n)\}$ denote the limit set of $\{x(n):n\geq 0\}$.

\begin{lemma}\label{Lem: non-convergence}
Let $H$ be a finite hypergraph. If $v\in \Lambda$ is an unstable equilibrium, then
    \[
      \mathds{P}\left[{\rm Lim}\{x(n)\}\cap (v+K)\neq\emptyset\right]=0.
    \]
In particular, $ \mathds{P}\left[\lim x(n)=v\right]=0$.
\end{lemma}

This formulation will be particularly important in the proof of Theorem \ref{thm-not-injective}.

\subsection{Non-unstable equilibria and Lyapunov functions}

The existence of non-unstable equilibria provides extra information on the behavior
of $F$. This is the content of the next lemma, which is a version 
for hypergraphs of \cite[Lemmas 3.1 and 3.2]{CL-14}.
Given $w \in \Delta_S$ and $\chi \in (0, \min_{i \in S} w_i]$, let
$\Delta^{w, \chi} = \{ v\in \mathbb{R}^m: v_i \geq \chi, \forall i \in S\}$.
Given $U\subset \Delta^{w, \chi}$, call $P:\Delta^{w, \chi}\to\R$
a {\em strict Lyapunov function for $U$} if it is strictly monotone along the integral curves of $F$
outside $U$.

\begin{lemma}\label{lemma-non-unstable}
Let $w$ be a non-unstable equilibrium. There exists a set $J = J(w, \chi)\subset w+K$
such that $P: \Delta^{w,\chi} \to \mathbb{R}$ given by $P(v) = \sum_{i=0}^m w_i \log v_i$
is a strict Lyapunov function for $J$.
\end{lemma}

\begin{proof}
Inside $\Delta^{w,\chi}$ the function $P$ is differentiable, with
$$
\frac{d}{dt}(P\circ v)=\sum_{i\in S}w_i\frac{1}{v_i}\frac{dv_i}{dt}=\sum_{i\in S}w_i\frac{\partial L}{\partial v_i}
=\sum_{i=1}^m w_i\frac{\partial L}{\partial v_i}=-1+\frac{1}{N}\sum_{I\in E}\frac{w_I}{v_I}\,\cdot
$$

Let $f:\Delta^{w,\chi}\to\mathbb R$ be given by $f(v)=-1+\frac{1}{N}\sum\limits_{I\in E}\frac{w_I}{v_I}$.
Observe that $f(w)=0$. We will show that $f(v)\geq 0$, with equality if and only if $v\in J$ (to be defined below).

\medskip
\noindent
{\sc Step 1:} $f$ is convex.

\medskip
Since $x>0\mapsto \frac{1}{x}$ is convex, each $v\in\Delta^{w,\chi}\mapsto \frac{w_I}{v_I}$ is convex,
thus $f$ is the sum of convex functions.

\medskip
\noindent
{\sc Step 2:} $w$ is a global minimum of $f$.

\medskip
Since $f$ is convex, it is enough to prove that $w$ is a local minimum of $f$.
Let $v=w+(\ve_1,\ldots,\ve_m)$ with $\ve_1,\ldots,\ve_m$ small enough.  Of course, $\ve_i\geq 0$ for $i\not\in S$.
Applying the inequality $\frac{x}{x+\ve}-1\geq -\frac{\ve}{x}$ for $x,x+\ve>0$, we have
\begin{align*}
&\ f(v)-f(w)=\frac{1}{N}\sum_{I\in E}\left[\frac{w_I}{v_I}-1\right]
\geq\frac{1}{N}\sum_{I\in E}-\frac{\ve_I}{w_I}=-\sum_{i=1}^m \ve_i\frac{1}{N}\sum_{I\ni i}\frac{1}{w_I}\\
&=-\sum_{i=1}^m \ve_i\left[1+\frac{\partial L}{\partial v_i}(w)\right]
=-\sum_{i=1}^m \ve_i \frac{\partial L}{\partial v_i}(w)\geq 0,
\end{align*}
since $\ve_i\frac{\partial L}{\partial v_i}(w)=0$ for $i\in S$, and $\ve_i\frac{\partial L}{\partial v_i}(w)\leq 0$
for $i\not\in S$. Hence $w$ is a local minimum of $f$.

\medskip
\noindent
{\sc Step 3:} The set of global minima of $f$ is $J=J(w,\chi):=\Delta^{w,\chi}\cap (w+K)$.

\medskip
The set of global minima of a convex function is convex. Thus if $v\in\Delta^{w,\chi}$ with $f(v)=f(w)$,
then $f(tv+(1-t)w)=tf(v)+(1-t)f(w)$ for all $t\in[0,1]$. Because $x>0\mapsto\frac{1}{x}$ is strictly convex,
we get $v_I=w_I$ for all $I\in E$, and so $v-w\in K$. This shows that the set of global minima of $f$ is
contained in $J$. Conversely, if $v\in J$, then $v_I=w_I$ for all $I\in E$, and so $f(v)=0$, which proves 
the reverse inclusion.   
\end{proof}

\section{Proof of Theorems \ref{Thm-injective} and \ref{thm-not-injective}}

In this section, we prove the main theorems.

\subsection{Proof of Theorem \ref{Thm-injective}}
We start observing that, since $K=\{0\}$, each $\Lambda_S$ is either empty or a singleton. 
By Theorem \ref{thm:limit set theorem}, we conclude that $\lim x(n)\in \Lambda$ is a singleton almost surely. 
By Lemma \ref{Lem: non-convergence}, there is at least one non-unstable equilibrium
$w\in\Lambda$. By Lemma \ref{lemma-non-unstable}, there is at most one non-unstable equilibrium. 
Therefore, $w$ is the only non-unstable equilibrium, and so $\lim x(n)=w$
almost surely.

\subsection{Proof of Theorem \ref{thm-not-injective}}

If $K=\{0\}$, then Theorem \ref{Thm-injective} provides the stronger result, 
hence we assume that $K\neq\{0\}$.
By Lemma \ref{Lem: non-convergence}, there is at least one non-unstable equilibrium
$w\in\Lambda$. 

\medskip
\noindent
{\sc The set $\mathfs J$:} We define $\mathfs J=\Delta^{m-1}\cap(w+K)$.

\medskip
Clearly, $\mathfs J\subset \Lambda$ and every $v\in \mathfs J$ is non-unstable.
We claim that $\mathfs J$ is the set of non-unstable equilibria.
By contradiction, assume that there is $v\in \Lambda\backslash \mathfs J$
non-unstable. Then $(v+K)\cap (w+K)=\emptyset$. However, by Lemma~\ref{lemma-non-unstable},
since $\Delta^{v,\chi} \cap \Delta^{w, \chi} \neq \emptyset$ for sufficiently small $\chi > 0$, 
the orbit by $F$ of a point at $\Delta^{v,\chi} \cap \Delta^{w, \chi}$ converges to 
both $J(v, \chi)$ and $J(w, \chi)$. In particular, it implies $(v + K) \cap (w +K) \neq \emptyset$, 
a contradiction. Thus $\mathfs J$ is the set of non-unstable equilibria. 
By Lemma~\ref{Lem: non-convergence}, ${\rm Lim}\left\{ x(n)\right\}  $
does not intersect the set of unstable equilibrium almost surely, 
and therefore ${\rm Lim}\{x(n)\}$ is contained in $\mathfs J$ almost surely.
This proves the first part of Theorem \ref{thm-not-injective}.

Now we prove the second part, so we restrict ourselves to the event that ${\rm Lim}\{x(n)\}\subset\Delta_{[m]}$. 
If this event has zero probability, there is nothing to prove. If it has positive probability, 
then $\Lambda_{[m]}\neq \emptyset$, so we can take $w$ above belonging
to $\Lambda_{[m]}$. In particular, $\Lambda_{[m]}\subset\mathfs J$. We will show that
$x(n)$ converges to a point of $\Lambda_{[m]}$ almost surely.
We begin calculating the dynamical nature of $F$ transversely to $\Lambda_{[m]}$.
Let $k={\rm dim}(K)$. Given $v\in\Lambda_{[m]}$, recall that $DF(v):T_v\Delta\to T_v\Delta$ 
has non-positive eigenvalues, due to the concavity of the Lyapunov function $L$. 

\begin{lemma}\label{lemma-neg-eig}
  For every $v\in \Lambda_{[m]}$, the derivative $DF(v)$ has
$m-1-k$ negative eigenvalues, and 0 is an eigenvalue with multiplicity $k$.
\end{lemma}  

\begin{proof}
 Consider the jacobian matrix $JF(v)$.
We claim that ${\rm ker}(JF(v))={\rm ker}(I(H))$. This will conclude the proof, since 
it implies that ${\rm ker}(DF(v))=K$, which has dimension $k$ and 
$DF(v)$ has non-positive eigenvalues. 
The vector field $F$ is zero on $v+{\rm ker}(I(H))$,
thus ${\rm ker}(I(H))\subset {\rm ker}(JF(v))$, so it is enough 
to prove that $I(H)$ and $JF(v)$ have the same rank.

Since $v\in \Delta_{[m]}$, we have $\frac{\partial L}{\partial v_i}(v)=0$ for all $i\in[m]$.
From equation (\ref{jacobian}), $JF(v)=\left(v_i \frac{\partial^2 L}{\partial v_i\partial v_j}(v)\right)_{i,j}$.
Letting $B={\rm Hess}(L(v))$, this implies that $JF(v)=DB$ where $D$ is an invertible diagonal matrix with entries $v_1,\ldots,v_m$, thus ${\rm rank}(JF(v))={\rm rank}(DB)={\rm rank}(B)$.
We now relate $B$ with $I(H)$. By equality (\ref{partial-derivative}),
$$
\frac{\partial L}{\partial v_i}=-1+\frac{1}{N}\sum_{I \ni i}\frac{1}{v_I} \ \ \Longrightarrow \ \ 
\frac{\partial^2 L}{\partial v_i\partial v_j}=-\frac{1}{N}\sum_{I \ni i, j}\frac{1}{(v_I)^2}\,\cdot
$$
Let $E$ be the matrix $I(H)$ with the line relative to the hyperedge $I$ multiplied by $\tfrac{1}{v_I}$.
Then ${\rm rank}(E)={\rm rank}(I(H))$ and the $(i,j)$--th entry of $E^tE$ is exactly
$\sum_{I \ni i, j}\frac{1}{(v_I)^2}$. Therefore $B=-\frac{1}{N}E^tE$, so that
${\rm rank}(B)={\rm rank}(E^tE)={\rm rank}(E)={\rm rank}(I(H))$. The conclusion is that
${\rm rank}(JF(v))={\rm rank}(I(H))$, which proves the lemma.
\end{proof}

\begin{remark}
We take the chance to use Lemma \ref{lemma-neg-eig} to correct a mistake in \cite{Lim-16},
where one of the authors claims that, for Pólya urns on graphs, a result similar to Lemma
\ref{lemma-neg-eig} holds simply because the restriction $L\restriction_{\Delta_{[m]}}$ is concave and 
$\Lambda_{[m]}$ is the set of global maxima of this restriction. These conditions are not enough 
to ensure that no other zero eigenvalue appears. 
\end{remark}

We are now in position to prove that, conditioned on the event that ${\rm Lim}\{x(n)\}\subset\Delta_{[m]}$,
 $x(n)$ converges to a point of $\Lambda_{[m]}$ almost surely.
The idea is to use the arguments of \cite{CL-14}, which we will recast the main ideas.
First, we introduce some notation.
Let $\{\Phi_t\}_{t\geq 0}$ be the semiflow induced by $F$ and
$\tau_n=\sum_{i=0}^n\gamma_i$. Let $\{X(t)\}_{t\geq 0}$ be the interpolation of $\{x(n)\}_{n\geq 0}$,
defined by $X(\tau_n)=x(n)$ and $X\restriction_{[\tau_n,\tau_{n+1}]}$ linear,
and let $d$ be the euclidean distance on $\Delta^{m-1}$.

\begin{theorem}[\cite{Benaim-99}]\label{theorem-random-versus-deterministic}
Almost surely, the interpolated process satisfies
$$
\sup_{T>0}\limsup_{t\to+\infty}\frac{1}{t}\log\left(\sup_{0\leq h\leq T}d\big(X(t+h),\Phi_h(X(t))\big)\right)
\leq -\frac{1}{2}\,\cdot
$$
\end{theorem}

\medskip
For Pólya urns on graphs, this result is stated in \cite[Lemma 4.1]{CL-14}.
The proof for hypergraphs is the same, following from shadowing techniques that relate the
rate of convergence of the interpolated process and the vector field, see \cite[Prop. 8.3]{Benaim-99}.
The right-hand side of the inequality is the log-convergence rate
$\frac{1}{2}\limsup\frac{\log\gamma_n}{\tau_n}=-\frac{1}{2}$.

We wish to show that $x(n)$ converges to a point of $\Lambda_{[m]}$ almost surely.
It is enough to prove that the interpolated process $X(t)$ satisfies this property.
Consider a foliation $\{\mathfs F_x\}_{x\in \Lambda_{[m]}}$ such that:
\begin{enumerate}[$\circ$]
\item $\mathfs F_x$ is a submanifold with $\mathfs F_x \pitchfork \Lambda_{[m]}$ at the single point $x$.
\item $x$ is a hyperbolic attractor for $F\restriction_{\mathfs F_x}$.
The speed of convergence depends on the negative eigenvalues of $DF(x)$.
\end{enumerate}
The existence of this foliation follows from the theory of invariant manifolds for normally hyperbolic sets,
see e.g. \cite[Theorem 4.1]{hirsch1977invariant}.
The leaves $\mathfs F_x$ depend smoothly on $x$. We let $\pi$ be the projection map
such that $\pi\restriction_{\mathfs F_x}\equiv x$.

Fix one realization of the process $x(n)$ whose limit set does not intersect $\partial\Delta^{m-1}$
(recall we conditioned on this event, assuming it has positive probability). Then the limit set of $X(t)$
also does not intersect $\partial\Delta^{m-1}$, and so $X(t)$ has an accumulation point $y\in \Lambda_{[m]}$.
Fix a compact neighborhood $V\subset \Lambda_{[m]}$ of $y$, and let
$U\subset\Delta$ be a neighborhood of $V$. Taking $U$ small enough,
the projection $\pi$ is 2--Lipschitz: 
\begin{equation}\label{pi-lipschitz}
d(\pi(x),\pi(y))\leq 2d(x,y),\forall\,x,y\in U.
\end{equation}
Fix a small parameter $\varepsilon>0$ and reduce $U$, if needed, so that
\begin{equation}\label{definition-neighborhood}
U=\{x\in\Delta:\pi(x)\in V\text{ and }d(x,\pi(x))<\varepsilon\}.
\end{equation}
Let $a=\max\{\lambda:\lambda\not=0\text{ is eigenvalue of }DF(x), x\in V\}$.
By Lemma \ref{lemma-neg-eig}, we have $a<0$, thus there is $C>0$ such that
$$
d(\Phi_t(x),\pi(x))\le C e^{at}d(x,\pi(x)), \forall\,x\in U,\forall\,t\ge 0.
$$

\begin{lemma}\label{lemma-iteration}
Assume that $X(t)\in U$. If $t,T$ are large enough, then
\begin{enumerate}[i)]
\item[{\rm (i)}] $d(\pi(X(t+T)),\pi(X(t)))<2e^{-\frac{t}{4}}$.
\item[{\rm (ii)}] $X(t+T)\in U$.
\end{enumerate}
\end{lemma}

\begin{proof}
Same as \cite[Lemma 4.4]{CL-14}, using Theorem \ref{theorem-random-versus-deterministic}
and estimates (\ref{pi-lipschitz}), (\ref{definition-neighborhood}).
\end{proof}

\medskip
Now, we repeat ipsis literis the arguments of \cite{CL-14}, see also \cite{Lim-16},
which consists of:
\begin{enumerate}[$\circ$]
\item Fix large $t,T$, and let $X_k=X(t+kT)$, $k\geq 0$. Applying Lemma \ref{lemma-iteration},
prove that $\pi(X_k)$ converges, say $\lim \pi(X_k)=\widetilde x\in V$.
\item Prove that  $\lim d(X_k,\pi(X_k))=0$, hence $\lim X_k=\widetilde x$.
\item Prove that $X(t)$ also converges to $\widetilde x$.
\end{enumerate}
Therefore $x(n)$ converges to $\widetilde x$. Since $V\subset\Lambda_{[m]}$,
the proof of Theorem \ref{thm-not-injective} is complete.

\section{Concluding remarks}\label{ss-concluding}

Recall from Section \ref{ss-examples} that, for the platonic solids, the kernel of $I(H)$ coincides
with $K$, and their dimensions are:
\begin{center}
\begin{tabular}{ |l||c| } 
\hline
Platonic solid & \hspace{.6cm} $\dim K$ \hspace{.6cm} \\
\hline \hline
Tetrahedron & $0$ \\ \hline
Cube &  4 \\ \hline
Octahedron & 2 \\ \hline
Icosahedron & $0$ \\ \hline
Dodecahedron & 8 \\ \hline
\end{tabular}
\end{center}
In all cases, the uniform measure $\left(\tfrac{1}{m},\ldots,\tfrac{1}{m}\right)$ 
is a {\em non-unstable} equilibrium.

\medskip
We are currently not able to prove the Conjecture~\ref{conj:hypergraph-convergence} due to the
behavior of $DF$ at $\partial\Delta^{m-1}$. As shown in Lemma \ref{lemma-neg-eig}, in the interior of the simplex
the eigenvalues associated to directions transverse to $K$ are negative, but
as we approach $\partial\Delta^{m-1}$ these eigenvalues can approach zero. Even worse:
in $\partial\Delta^{m-1}$ the rank of $DF$ can decrease, thus creating zero eigenvalues 
in directions transverse to $K$.
For example, consider the Pólya urn in the cube, with the following enumeration of vertices:

\begin{figure}[!h]
 \centering
 \begin{tikzpicture}[cube/.style={thick,black}, cube hidden/.style={thick,dashed}]
	\draw[cube, line join=round] (2,0,0) -- (2,2,0) -- (2,2,2) -- (2,0,2) -- cycle;
  \draw[cube, line join=round] (0,2,0) node[anchor=south]{4} -- (2,2,0) node[anchor=west]{3}-- (2,2,2) -- (0,2,2)  -- cycle;
	\draw[cube, line join=round] (0,0,2) node[anchor=east]{5} -- (0,2,2) node[anchor=east]{1} -- (2,2,2)  node[anchor=west]{2} -- (2,0,2) node[anchor=west]{6}-- cycle;
	\draw[cube hidden, line join=round] (0,0,0) -- (2,0,0) node[anchor=west]{7}; 
	\draw[cube hidden, line join=round] (0,0,0) node[anchor=north]{8} -- (0,2,0);
	\draw[cube hidden,line join=round] (0,0,0) -- (0,0,2);
\end{tikzpicture}
\end{figure}

We have $\mathfs J=\Delta^7\cap (w+K)$, where $w=\left(\tfrac{1}{8},\ldots,\tfrac{1}{8}\right)$
and $K$ has dimension 4. On $w+K$, we have $v_I=\tfrac{1}{2}$ for all $I\in E$, hence
$$
\tfrac{\partial^2L}{\partial v_i\partial v_j}=-\tfrac{1}{6}\tfrac{1}{\left(\tfrac{1}{2}\right)^2}\#\{I\in E:i,j\in I\}=
-\tfrac{2}{3}\#\{I\in E:i,j\in I\}\,\cdot
$$
The set $\mathfs J$ has points with different
behavior for $F$:
\begin{enumerate}[$\circ$]
\item For $v=\left(0,\tfrac{1}{4},0,\tfrac{1}{4},\tfrac{1}{8},\tfrac{1}{8},\tfrac{1}{8},\tfrac{1}{8}\right)$, the matrix
$$
JF(v)=-\tfrac{1}{12}
\begin{bmatrix}
0 & 0 & 0 & 0 & 0 & 0 & 0 & 0 \\
4 & 6 & 4 & 2 & 2 & 4 & 2 & 0 \\
0 & 0 & 0 & 0 & 0 & 0 & 0 & 0 \\
4 & 2 & 4 & 6 & 2 & 0 & 2 & 4 \\
2 & 1 & 0 & 1 & 3 & 2 & 1 & 2 \\
1 & 2 & 1 & 0 & 2 & 3 & 2 & 1 \\ 
0 & 1 & 2 & 1 & 1 & 2 & 3 & 2 \\
1 & 0 & 1 & 2 & 2 & 1 & 2 & 3 \\
\end{bmatrix}
$$
has rank 4, hence no new eigenvectors of 0 are created. 
\item For $v=\left(\tfrac{1}{4},0,0,\tfrac{1}{4},0,\tfrac{1}{4},\tfrac{1}{4},0\right)$, the matrix
$$
JF(v)=-\tfrac{1}{6}
\begin{bmatrix}
3 & 2 & 1 & 2 & 2 & 1 & 0 & 1 \\
0 & 0 & 0 & 0 & 0 & 0 & 0 & 0 \\
0 & 0 & 0 & 0 & 0 & 0 & 0 & 0 \\
2 & 1 & 2 & 3 & 1 & 0 & 1 & 2 \\
0 & 0 & 0 & 0 & 0 & 0 & 0 & 0 \\
1 & 2 & 1 & 0 & 2 & 3 & 2 & 1 \\
0 & 1 & 2 & 1 & 1 & 2 & 3 & 2 \\
0 & 0 & 0 & 0 & 0 & 0 & 0 & 0 \\
\end{bmatrix}
$$
has rank 3, and five eigenvectors of 0 in $\Gamma$. Hence,
a new eigenvalue 0 was created.
\end{enumerate}
We believe that in this case, the limiting distribution
is fully supported on $\mathfs J$ and depends on the initial condition $x(0)$.
See the simulations in Figure \ref{simulations}, each consisting of 10000 iterations. 
The resulting points $\{(n,x(n)):0\leq n\leq 10000\}$ are plotted as a graph.
Notice that, as $B_1(0)$ grows,
$\lim x_1(n)$ tends to dominate the other limits. 
\begin{figure}
\includegraphics[scale=0.339]{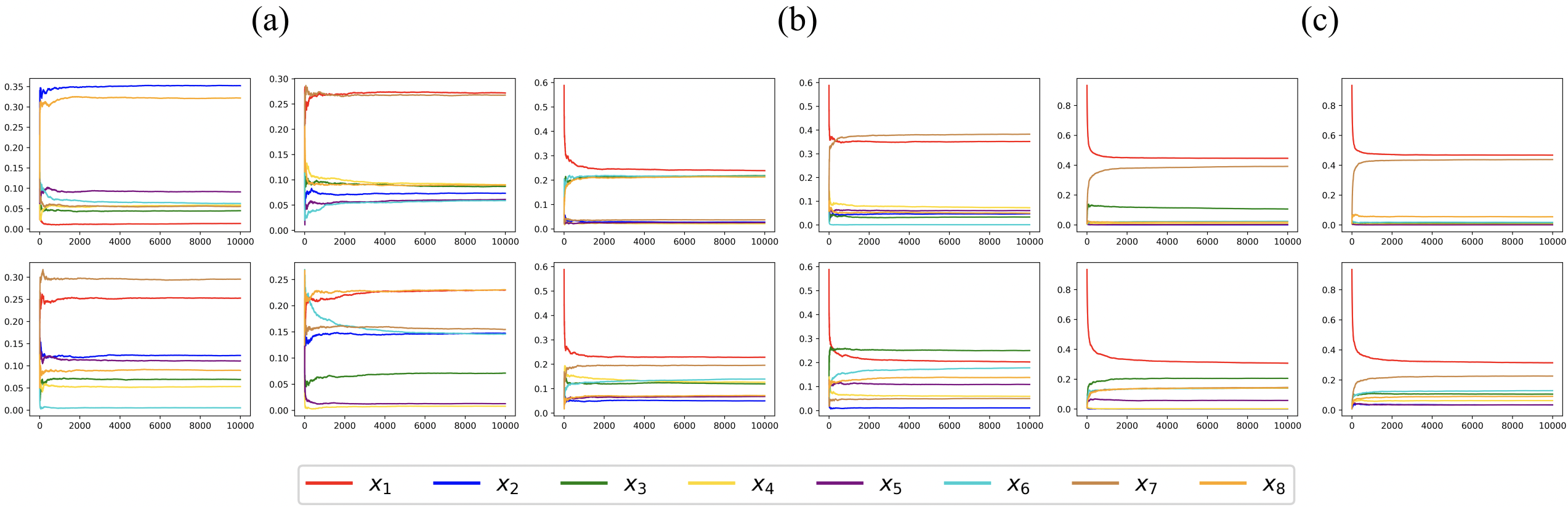}
\caption{(a) Four simulations on the cube with initial condition $(1,1,1,1,1,1,1,1)$. 
(b) Four simulations on the cube with initial condition $(10,1,1,1,1,1,1,1)$. 
(c) Four simulations on the cube with initial condition $(100,1,1,1,1,1,1,1)$.
The almost vertical components in the graphs in the first
iterations are due to large oscillations in the beginning of the process.}
\label{simulations}
\end{figure}

\medskip
Certainly, the hypothesis of Theorem \ref{thm-not-injective} is not global,
as the next lemma shows. Given a vertex $i\in[m]$, we let $E^i = \{I \in E: i \in I\}$
denote the set of hyperedges containing $i$. 

\begin{lemma}\label{lemma-leaf}
Assume that there are $I\in E$ and $i,j\in I$ such that $E^i=\{I\}$ and $|E^j|>1$. 
If $w\in\Lambda$ with $w_i>0$, then $w$ is unstable.
In particular, the limit set of $x(n)$ converges to $\partial\Delta^{m-1}$ almost surely.
\end{lemma}

\begin{proof}
Since $w_i>0$, we have $\frac{\partial L}{\partial v_i}(w)=0$, 
then 
$$
0=\frac{\partial L}{\partial v_i}(w)=-1+\frac{1}{N}\cdot\frac{1}{w_I}\ \ \Longrightarrow \ \ 
w_I=\frac{1}{N}\,\cdot
$$
Now,
$$ \frac{\partial L}{\partial v_j}(w)=-1+\frac{1}{N}\left(\frac{1}{w_I}+\sum_{\substack{J\ni j\atop{J\neq I}}}\frac{1}{w_J}\right)=\frac{1}{N}\sum_{J\ni j\atop{J\neq I}}\frac{1}{w_J}>0. $$
Since $w_j\frac{\partial L}{\partial v_j}(w)=0$, we have $w_j=0$, thus
proving that $w$ is unstable. Therefore, every non-unstable equilibrium $w\in\Lambda$
must satisfy $w_i=0$. Applying Lemma \ref{Lem: non-convergence}, the proof is complete.
\end{proof}

This shows that there is a large class of hypergraphs that require a finer analysis 
of equilibria in $\partial\Delta^{m-1}$.

\section{Acknowledgements}

This work started during the undergraduate research program
``Jornadas de Pesquisa para Graduação'', held at ICMC-USP in January 2023, 
which was partially supported by Centro de Ciências Matemáticas Aplicadas à Industria
(CeMEAI - CEPID) under FAPESP Grant \#2013/07375-0.
The authors are grateful to Ali Tahzibi, Gracyella Salcedo, Guilherme Silva,
Marcelo Tabarelli and Rafael Zorzetto for the discussions and support, and the anonymous referees
for the suggestions that greatly improved the quality of the text.
PA was supported by CAPES. MB was supported by CAPES and CNPq. 
YL was supported by CNPq/MCTI/FNDCT project 406750/2021-1; 
FUNCAP grant UNI-0210-00288.01.00/23; and Instituto Serrapilheira,
grant ``Jangada Din\^{a}mica: Impulsionando Sistemas Din\^{a}micos na Regi\~{a}o Nordeste''.

\appendix

\section{Non-convergence to unstable equilibria}\label{sec: non-convergence}

In this appendix we prove Lemma \ref{Lem: non-convergence},
with methods similar to \cite[Lemma 5.2]{BBCL-15}. We begin with a simple adaptation of 
the second Borel-Cantelli lemma, that holds for events that are not necessarily independent.

\begin{lemma}\label{lem:borelcantelli}
    Let $\{E_i\}_{i\geq 0}$ be a sequence of events, not necessarily independent. 
    If
    \begin{equation*}\label{}
        \sum_{i=n+1}^{\infty}\mathds{P}\left[E_i|E_n^\mathsf{c},\ldots,E_{i-1}^{\mathsf{c}} \right] =\infty \ \text{ for all } n \in \mathbb{N}
        ,
    \end{equation*}
    then $\mathds{P}\left[ \limsup\limits_{i \to \infty} E_{i} \right]  = 1$.
\end{lemma}
\begin{proof}
    Since 
\begin{align*}
    1-\mathds{P}\left[\limsup_{i\to\infty}E_i\right] &= \mathds{P}\left[\liminf_{i\to\infty}E_i^\mathsf{c}\right]
    =\mathds{P}\left[\bigcup_{n=1}^\infty \bigcap_{i=n}^\infty E_i^\mathsf{c}\right]
    =\lim_{n\to\infty}\mathds{P}\left[\bigcap_{i=n}^\infty E_i^\mathsf{c}\right],
\end{align*}
it is enough to prove that $\mathds{P}\left[\bigcap_{i=n}^\infty E_i^\mathsf{c}\right] = 0$ for every $n \in \mathbb{N}$.
Since
\begin{align*}
   & \mathds{P}\left[\bigcap_{i=n}^\infty E_i^\mathsf{c}\right] = \mathds{P}[E_n^\mathsf{c}]\prod_{i=n+1}^\infty \mathds{P}[E_i^\mathsf{c}|E_n^\mathsf{c},\ldots,E_{i-1}^\mathsf{c}]
    =\mathds{P}[E_n^\mathsf{c}]\prod_{i=n+1}^\infty\left(1-\mathds{P}[E_i|E_n^\mathsf{c},\ldots,E_{i-1}^\mathsf{c}]\right)\\
    &\leq \mathds{P}[E_n^\mathsf{c}]\prod_{i=n+1}^\infty e^{-\mathds{P}[E_i|E_n^\mathsf{c},\ldots,E_{i-1}^\mathsf{c}]}
    = \mathds{P}[E_n^\mathsf{c}]\cdot e^{-\sum_{i=n+1}^\infty \mathds{P}[E_i|E_n^\mathsf{c},\ldots,E_{i-1}^\mathsf{c}]}=0,
\end{align*}
since $\sum_{i=n+1}^{\infty}\mathds{P}[E_i|E_n^\mathsf{c},\ldots,E_{i-1}^\mathsf{c}]=\infty$.
\end{proof}

Recall that $E^i = \{I \in E: i \in I\}$ is the set of hyperedges containing $i$. 

\begin{lemma}\label{Lem: neighborhood N}
Let $v \in \Lambda$ with $v_1=0$ and $\tfrac{\partial L}{\partial v_i (v)} > 3\delta$.
Then there exists a neighborhood $\mathcal{N}$ of $v + K$ not containing any other equilibrium,
an element $u \in \mathcal{N}$ and $\epsilon_0 > 0$ such that
\begin{enumerate}[{\rm (1)}]
\item $\frac{\partial L}{\partial v_i} (u) > 3\delta+\frac{|E^1|\epsilon_0}{N}$, and
\item for all $w \in \mathcal{N}$ and $I \in E^1$ it holds
\[
\frac{1}{w_I}>\frac{1}{u_I}-\epsilon_0.
    \]
\end{enumerate}
\end{lemma}

\begin{proof}
  Fix $\epsilon_0 \in \left(0, \left(\frac{\partial L}{\partial v_i}(v) -3\delta\right)\frac{N}{|E^1|}\right)$,
  then $\frac{\partial L}{\partial v_i}(v) > 3\delta + \frac{|E^1|\epsilon_0}{N}$.
  Since $\frac{\partial L}{\partial v_i}$ is continuous and equal to $\frac{\partial L}{\partial v_i} (v)$
  along $v + K$, there exists a neighborhood $\mathcal{N}$ of $v + K$ small enough such that
  it does not contain any other equilibrium and also satisfies condition (1).

  Given $I \in E^1$, the function $w \in \overline{\mathcal{N}} \mapsto \frac{1}{w_I}$ is uniformly continuous and therefore there exists $\overline{\delta}_I$ such that
$$
|w-\widetilde{w}|<\overline{\delta}_I \implies \left|\frac{1}{w_I}-\frac{1}{\widetilde{w}_I}\right|<\epsilon_0, 
\ \forall w,\widetilde{w} \in \overline{\mathcal{N}}.
$$
Thus, if we take $\diam(\mathcal{N})< \max\limits_{I \in E^1} \overline{\delta}_I$,
then condition (2) is also satisfied.
\end{proof}

\begin{proof}[Proof of Lemma \ref{Lem: non-convergence}]
Since $v$ is an unstable equilibrium, there exists $i \in [m]$ such that $v_i = 0$
and $\frac{\partial L}{\partial v_i}(v) > 0$. Without loss of generality, assume $i = 1$.
Since ${\rm Lim}\{x(n)\}$ is a connected subset of $\Lambda$ (Theorem~\ref{thm:limit set theorem}),
if $v \in {\rm Lim}\{x(n)\}$, then ${\rm Lim}\{x(n)\} \subset v + K$.
Firstly, we claim that
\begin{equation}\label{eq:balls}
\mathds{P}\left[\lim\limits_{n\to\infty}B_1(n)=\infty\right]=1.
 \end{equation}
Fix an edge $I \in E^1$ and let $Z_n$ be the event that the vertex 1 is the chosen from hyperedge
$I$ at step $n+1$. Because $1 \leq B_i(n) \leq N_0 + nN$ for all $i \in [m]$,
  \[
    \mathds{P}\left[Z_n | Z_k^{\mathsf{c}}, \ldots, Z_{n-1}^{\mathsf{c}} \right]\geq \dfrac{1}{|I|(N_0+nN)}
  \]
  for every $k<n$. Then (\ref{eq:balls}) follows by Lemma~\ref{lem:borelcantelli}.
  
Let $\delta > 0$ and $\mathcal{N}$ be as in Lemma~\ref{Lem: neighborhood N},
and fix $B$ large enough (to be specified later).
We define the event $\mathcal{Y}_n = \{x(k) \in \mathcal{N}, \forall k \geq n\} \cap \{B_1(n) \geq B \}$.
Observing that $\{{\rm Lim}\{x(n)\}\subset v+K\} \subset \bigcup_{m \geq n} \mathcal{Y}_{m}$
for every $n>0$, it  is enough to prove that there is $n_0 > 0$ such that 
$$
\mathds{P}\left[ \mathcal{Y}_{n} \right] = 0, \ \forall n > n_0.
$$	
Let $\mathcal{G}_n = \mathcal{F}_n \cap \mathcal{Y}_n$.
Define $c = \frac{\delta(1 + 2\delta)}{1 + \frac{3}{2}\delta} > 0$.
We claim that, if $B$ is large enough, then there exists $n_0 > 0$ such that
\begin{equation}\label{eq: log expectation}
\mathds{E}\left[\log x_1((1+\delta)n)|\mathcal{G}_n\right] \geq \log x_1(n) + \frac{1}{2} \log(1+c), \ \forall n > n_0.
\end{equation}
Once we have this, suppose there exists $n>n_0$ such that
$\mathds{P}[\mathcal{Y}_{n}] > 0$ and define $T_k = (1+\delta)^kn$ and $X_k = \log x_1(T_k)$. By (\ref{eq: log expectation}),
$$
\mathds{E}\left[X_{k+1} | \mathcal{G}_{n}\right] = \mathds{E}\left[\mathds{E}\left[X_{k+1} | \mathcal{G}_{T_k}\right]| \mathcal{G}_{n}\right] \geq \mathds{E}\left[X_k|\mathcal{G}_n\right] + \frac{1}{2}\log(1+c)
$$
and so by induction
$$
\mathds{E}\left[X_{k} | \mathcal{G}_{n}\right] \geq \mathds{E}[X_0|\mathcal{G}_n] + \frac{k}{2}\log(1+c)
\geq -\log\left(N_0 + nN\right) + \frac{k}{2}\log(1+c).
$$
Since the left hand side is nonpositive, we obtain a contradiction when $k$ is large enough. 

Now we prove (\ref{eq: log expectation}). Let $t \in \{n+1, \ldots, (1+\delta)n\}$. 
Restricted to $\mathcal{Y}_{n_0}$, for every $I \in E^1$ we have 
    \begin{align*}
   &\ \mathds{P}[1 \text{ is chosen in } I \text{ at step } t] = \dfrac{B_1(t-1)}{B_I(t-1)}
    \geq \dfrac{B_1(n)}{N_0+(t-1)N}\cdot\dfrac{1}{x_I(t-1)}\\
    &\geq \dfrac{B_1(n)}{N_0+(t-1)N}\left(\dfrac{1}{u_I}-\epsilon_0\right).
    \end{align*}
    Define a family of independent Bernoulli random variables $\{E_{t,I}\}$, $t=n+1,\ldots,(1+\delta)n$, $I \in E^1$ such that \[
    \mathds{P}[E_{t,I}=1] = \dfrac{B_1(n)}{N_0+(t-1)N}\left(\dfrac{1}{u_I}-\epsilon_0\right).
    \]
    Now couple $\{E_{t,I}\}$ to our model as follows: if $E_{t,I}=1$, then $1$ is chosen in $I$ at step $t$. Then
    \begin{align*}
       & \mathds{E}\left[\sum_{\substack {n+1\leq t \leq (1+\delta)n \\ I \in E^1}} E_{t,I}\right] = \sum_{I \in E^1}\left(\sum_{t=n+1}^{(1+\delta)n}\dfrac{B_1(n)}{N_0+(t-1)N}\left(\dfrac{1}{u_I}-\epsilon_0\right)\right)\\
        &=B_1(n)\left[\sum_{I\in E^1}\left(\dfrac{1}{u_I}-\epsilon_0\right)\right]
        \left[\sum_{t=n+1}^{(1+\delta)n}\dfrac{1}{N_0+(t-1)N}\right]
         \geq B_1(n) (1+3\delta)\log\left(1+\dfrac{\delta nN}{N_0+nN}\right).
    \end{align*}
    If $n_0$ is large enough (such that $\frac{n_0 N}{N_0+n_0 N} > \frac{1}{1+\frac{1}{2}\delta}$), we get \[
    \mathds{E}\left[\sum_{\substack {n+1\leq t \leq (1+\delta)n \\ I \in E^1}} E_{t,I}\right] \geq B_1(n) \dfrac{\delta (1+3\delta)}{1+\frac{3}{2}\delta}.
    \]
    By Chernoff bounds, if $\epsilon_1>0$, then there is $B_0$ large enough such that\[
    \mathds{P}\left[\sum_{\substack {n+1\leq t \leq (1+\delta)n \\ I \in E^1}} E_{t,I} > B_1(n)\dfrac{\delta (1+2\delta)}{1+\frac{3}{2}\delta} \right] > 1-\epsilon_1
    \]
    for every $B_1(n)>B_0$. Whenever the previous event holds, the coupling gives us
    \begin{equation*}
        B_1((1+\delta)n)-B_1(n) \geq \sum_{\substack {n+1\leq t \leq (1+\delta)n \\ I \in E^1}} E_{t,I} > B_1(n)\dfrac{\delta (1+2\delta)}{1+\frac{3}{2}\delta},
    \end{equation*}
    thus, $\mathds{P}[x_1((1+\delta)n)>x_1(n)(1+c)|\mathscr{G}_n]>1-\epsilon_1$. Now, because $x_1((1+\delta)n)>\frac{x_1(n)}{1+\delta}$, we get 
    \begin{equation*}
        \mathds{E}[\log x_1((1+\delta)n)|\mathscr{G}_n] > (1-\epsilon_1)\log(x_1(n)(1+c))+\epsilon_1\log\left(\dfrac{x_1(n)}{1+\delta}\right).
    \end{equation*}
    For large $B_0$, $\epsilon_1$ can be made small enough such that
    \begin{equation*}
    	\mathds{E}[\log x_1((1+\delta)n)|\mathscr{G}_n] > \log x_1(n)+\dfrac{1}{2}\log (1+c),
    \end{equation*}
	thus proving (\ref{eq: log expectation}), and hence Lemma \ref{Lem: non-convergence}.
  \end{proof}

\bibliographystyle{alpha}
\bibliography{bibliography}{}

\newcommand{\etalchar}[1]{$^{#1}$}
\begin{thebibliography}{vdHHKR16}

\bibitem[ACG19]{ACG-19}
Giacomo Aletti, Irene Crimaldi, and Andrea Ghiglietti.
\newblock Networks of reinforced stochastic processes: asymptotics for the
  empirical means.
\newblock {\em Bernoulli}, 25(4B):3339--3378, 2019.

\bibitem[ACG20]{Aletti-20}
Giacomo Aletti, Irene Crimaldi, and Andrea Ghiglietti.
\newblock Interacting reinforced stochastic processes: statistical inference
  based on the weighted empirical means.
\newblock {\em Bernoulli}, 26(2):1098--1138, 2020.

\bibitem[AMR16]{Antunovic-16}
Ton\'{c}i Antunovi\'{c}, Elchanan Mossel, and Mikl\'{o}s~Z. R\'{a}cz.
\newblock Coexistence in preferential attachment networks.
\newblock {\em Combin. Probab. Comput.}, 25(6):797--822, 2016.

\bibitem[BBCL15]{BBCL-15}
Michel Benaim, Itai Benjamini, Jun Chen, and Yuri Lima.
\newblock A generalized {P}{\'o}lya's urn with graph-based interactions.
\newblock {\em Random Structures \& Algorithms}, 46(4):614--634, 2015.

\bibitem[BC22]{BC-22}
Jacopo Borga and Benedetta Cavalli.
\newblock Quenched law of large numbers and quenched central limit theorem for
  multiplayer leagues with ergodic strengths.
\newblock {\em Ann. Appl. Probab.}, 32(6):4398--4425, 2022.

\bibitem[Ben96]{Benaim-96}
Michel Benaim.
\newblock A dynamical system approach to stochastic approximations.
\newblock {\em SIAM J. Control Optim.}, 34(2):437--472, 1996.

\bibitem[Ben99]{Benaim-99}
Michel Bena\"{\i}m.
\newblock Dynamics of stochastic approximation algorithms.
\newblock In {\em S\'{e}minaire de {P}robabilit\'{e}s, {XXXIII}}, volume 1709
  of {\em Lecture Notes in Math.}, pages 1--68. Springer, Berlin, 1999.

\bibitem[CDPLM19]{CDLM-19}
Irene Crimaldi, Paolo Dai~Pra, Pierre-Yves Louis, and Ida~G. Minelli.
\newblock Synchronization and functional central limit theorems for interacting
  reinforced random walks.
\newblock {\em Stochastic Process. Appl.}, 129(1):70--101, 2019.

\bibitem[CH21]{CH-21}
Yannick Couzini\'{e} and Christian Hirsch.
\newblock Weakly reinforced {P}\'{o}lya urns on countable networks.
\newblock {\em Electron. Commun. Probab.}, 26:Paper No. 35, 10, 2021.

\bibitem[CJ22]{CJ-22}
Marcelo Costa and Jonathan Jordan.
\newblock Phase transitions in non-linear urns with interacting types.
\newblock {\em Bernoulli}, 28(4):2546--2562, 2022.

\bibitem[CL14]{CL-14}
Jun Chen and Cyrille Lucas.
\newblock A generalized p{\'o}lya's urn with graph based interactions:
  convergence at linearity.
\newblock {\em Electronic Communications in Probability}, 19:1--13, 2014.

\bibitem[DG23]{DG23}
Qionghai Dai and Yue Gao.
\newblock {\em Hypergraph computation}.
\newblock Artificial {I}ntelligence: {F}oundations, {T}heory, and {A}lgorithms.
  Springer, Singapore, 2023.

\bibitem[DLZG20]{DLZG20}
Donglin Di, Shengrui Li, Jun Zhang, and Yue Gao.
\newblock Ranking-{B}ased {S}urvival {P}rediction on {H}istopathological
  {W}hole-{S}lide {I}mages.
\newblock In {\em Medical Image Computing and Computer Assisted Intervention --
  MICCAI $2020$}, pages 428--438. Springer International Publishing, 2020.

\bibitem[Dur10]{Durrett-book}
Rick Durrett.
\newblock {\em Probability: theory and examples}, volume~31 of {\em Cambridge
  Series in Statistical and Probabilistic Mathematics}.
\newblock Cambridge University Press, Cambridge, fourth edition, 2010.

\bibitem[DZF{\etalchar{+}}23]{DZFZJDG23}
Donglin Di, Changqing Zou, Yifan Feng, Haiyan Zhou, Rongrong Ji, Qionghai Dai,
  and Yue Gao.
\newblock Generating {H}ypergraph-{B}ased {H}igh-{O}rder {R}epresentations of
  {W}hole-{S}lide {H}istopathological {I}mages for {S}urvival {P}rediction.
\newblock {\em IEEE Transactions on Pattern Analysis and Machine Intelligence},
  45(5):5800--5815, 2023.

\bibitem[GWT{\etalchar{+}}12]{GWTJD12}
Yue Gao, Meng Wang, Dacheng Tao, Rongrong Ji, and Qionghai Dai.
\newblock 3-{D} {O}bject {R}etrieval and {R}ecognition {W}ith {H}ypergraph
  {A}nalysis.
\newblock {\em IEEE Transactions on Image Processing}, 21(9):4290--4303, 2012.

\bibitem[HHK21]{HHK-21}
Christian Hirsch, Mark Holmes, and Victor Kleptsyn.
\newblock Absence of {WARM} percolation in the very strong reinforcement
  regime.
\newblock {\em Ann. Appl. Probab.}, 31(1):199--217, 2021.

\bibitem[HHK23]{HHK-23}
Christian Hirsch, Mark Holmes, and Victor Kleptsyn.
\newblock Infinite {WARM} graphs {III}: strong reinforcement regime.
\newblock {\em Nonlinearity}, 36(6):3013--3042, 2023.

\bibitem[HPS77]{hirsch1977invariant}
M.~W. Hirsch, C.~C. Pugh, and M.~Shub.
\newblock {\em Invariant manifolds}.
\newblock Lecture Notes in Mathematics, Vol. 583. Springer-Verlag, Berlin-New
  York, 1977.

\bibitem[KMS22]{KMS-22}
Daniel Kious, C\'{e}cile Mailler, and Bruno Schapira.
\newblock The trace-reinforced ants process does not find shortest paths.
\newblock {\em J. \'{E}c. polytech. Math.}, 9:505--536, 2022.

\bibitem[Lim16]{Lim-16}
Yuri Lima.
\newblock Graph-based p{\'o}lya’s urn: completion of the linear case.
\newblock {\em Stochastics and Dynamics}, 16(02):1660007, 2016.

\bibitem[Pem92]{Pemantle-92}
Robin Pemantle.
\newblock Vertex-reinforced random walk.
\newblock {\em Probab. Theory Related Fields}, 92(1):117--136, 1992.

\bibitem[RJY{\etalchar{+}}21]{RJY+21}
Ding Ruan, Shuyi Ji, Chenggang Yan, Junjie Zhu, Xibin Zhao, Yuedong Yang, Yue
  Gao, Changqing Zou, and Qionghai Dai.
\newblock Exploring complex and heterogeneous correlations on hypergraph for
  the prediction of drug-target interactions.
\newblock {\em Patterns}, 2(12), 2021.

\bibitem[RPP22]{RPP-22}
Rafael~A. Rosales, Fernando P.~A. Prado, and Benito Pires.
\newblock Vertex reinforced random walks with exponential interaction on
  complete graphs.
\newblock {\em Stochastic Process. Appl.}, 148:353--379, 2022.

\bibitem[SAG22]{AGS-22}
Somya Singh, Fady Alajaji, and Bahman Gharesifard.
\newblock A finite memory interacting {P}\'{o}lya contagion network and its
  approximating dynamical systems.
\newblock {\em SIAM J. Control Optim.}, 60(2):S347--S369, 2022.

\bibitem[Sah16]{Sahasrabudhe-16}
Neeraja Sahasrabudhe.
\newblock Synchronization and fluctuation theorems for interacting {F}riedman
  urns.
\newblock {\em J. Appl. Probab.}, 53(4):1221--1239, 2016.

\bibitem[SP00]{SP-2000}
B.~Skyrms and R.~Pemantle.
\newblock A dynamic model of social network formation.
\newblock {\em Proceedings of the National Academy of Sciences of the United
  States of America}, 97(16):9340--9346, 2000.

\bibitem[vdHHKR16]{Hofstad-16}
Remco van~der Hofstad, Mark Holmes, Alexey Kuznetsov, and Wioletta Ruszel.
\newblock Strongly reinforced {P}\'{o}lya urns with graph-based competition.
\newblock {\em Ann. Appl. Probab.}, 26(4):2494--2539, 2016.

\bibitem[YTW12]{YTW12}
Jun Yu, Dacheng Tao, and Meng Wang.
\newblock Adaptive {H}ypergraph {L}earning and its {A}pplication in {I}mage
  {C}lassification.
\newblock {\em IEEE Transactions on Image Processing}, 21(7):3262--3272, 2012.

\bibitem[ZLZ{\etalchar{+}}18]{ZLZJG18}
Zizhao Zhang, Haojie Lin, Xibin Zhao, Rongrong Ji, and Yue Gao.
\newblock Inductive {M}ulti-{H}ypergraph {L}earning and {I}ts {A}pplication on
  {V}iew-{B}ased 3{D} {O}bject {C}lassification.
\newblock {\em IEEE Transactions on Image Processing}, 27(12):5957--5968, 2018.

\end{thebibliography}

\end{document}